\documentclass[11pt]{article}
\usepackage{fullpage}
\usepackage{enumerate} 
\usepackage{graphicx}
\usepackage{subfig}

\usepackage{xcolor}

\newcommand{\sdi}{{\mathbf{d}}}

\newcommand{\I}{\textbf{I}}

\newcommand{\seg}{\widehat{\textbf{g}}}
\newcommand{\srg}{\textbf{g}}

\newcommand{\beps}{{\boldsymbol{\epsilon}}}

\newcommand{\sVtori}{{\tilde{\textbf{v}}}}
\newcommand{\sVtorip}{{\tilde{\textbf{v}}}^{-\frac{1}{2}}}

\newcommand{\y}{{\textbf{y}}}

\newcommand{\z}{\textbf{z}}

\usepackage{algpseudocode}%
\newcommand{\wi}{\textbf{v}}
\newcommand{\m}{\textbf{m}}
\newcommand{\bxi}{\boldsymbol{\xi}}
\newcommand{\btheta}{{\boldsymbol{\theta}}}
\newcommand{\balpha}{{\boldsymbol{\alpha}}}

\newcommand{\x}{{\mathbf{x}}}
\newcommand{\E}{{\mathbb{E}}}

\newcommand{\ADAM}{{\textsc{ADAM}$^3$}}
\newcommand{\SADAM}{{\textsc{Adam}}}
\newcommand{\SSADAM}{{\textsc{S-Adam}}}

\usepackage{wrapfig}
\usepackage{varwidth}
\usepackage{amsmath,graphicx}

\usepackage{array}
\usepackage[outline]{contour}
\usepackage{amsthm}
\usepackage{subfig,placeins} 
\usepackage{graphicx,color,float,epstopdf}
\usepackage{fix-cm}
\usepackage{amssymb}
\usepackage{arydshln}
\usepackage{tabularx}
\usepackage{mathtools}
\usepackage{graphicx}
\usepackage{psfrag}
\usepackage{epsfig}
\usepackage{graphics,bm,paralist,ifthen,color}
\usepackage{array}
\usepackage{calc}
\usepackage{longtable}
\usepackage{mathrsfs}
\usepackage{float}
\usepackage{lipsum}
\usepackage{amsmath}
\usepackage{thmtools,thm-restate}
\usepackage{amssymb}
\usepackage{lmodern}
\usepackage[utf8]{inputenc}
\usepackage{thmtools}
\usepackage{thm-restate}
\usepackage{hyperref}
\usepackage{amsthm}

\usepackage{cleveref}
\def\BState{\State\hskip-\ALG@thistlm}
\makeatother

\newcommand{\R}{{\mathbb{R}}}
\providecommand{\keywords}[1]
{
  \small	
  \textbf{\textit{Keywords---}} #1
}

\usepackage[ruled,vlined]{algorithm2e}
\usepackage{tikz-cd}

\newtheorem{lemma}{Lemma}

\newtheorem{assum}{Assumption}

\newtheorem{deff}{Definition}
\newtheorem{rmk}{Remark}
\newtheorem{thm}{Theorem}
\newtheorem{cor}{Corollary}
\usepackage{amsmath,varwidth,array,ragged2e}
\newcolumntype{L}{>{\varwidth[c]{\linewidth}}l<{\endvarwidth}}
\newcolumntype{M}{>{$}l<{$}}
\def\x{{\mathbf x}}

\title{Solving a Class of Non-Convex Min-Max Games using Adaptive Momentum Methods\thanks{This arXiv submission includes the details of the proofs for the paper accepted for publication in the proceeding of the $46^{th}$ International Conference on Acoustics, Speech, and Signal Processing (ICASSP).}}

\newcommand*{\email}[1]{\texttt{#1}}
\author{
$^{\star}$Babak Barazandeh
\\
\email{bbarazandeh@splunk.com}
\and 
$^{\dagger}$Davoud Ataee Tarzanagh
\\
\email{tarzanagh@ufl.edu}
\and
$^{\dagger}$George Michailidis
\\
\email{gmichail@ufl.edu}
}
\date{
$^{\star}$Splunk, $^{\dagger}$University of Florida
}

\begin{document}
%
\maketitle

\begin{abstract}
Adaptive momentum methods have recently attracted a lot of attention for training of deep neural networks. They use an exponential moving average of past gradients of the objective function to update both search directions and learning rates. However, these methods are not suited for solving min-max optimization problems that arise in training  generative adversarial networks. In this paper, we propose an adaptive momentum min-max algorithm that generalizes adaptive momentum methods to the non-convex min-max regime. Further, we establish  non-asymptotic rates of convergence for the proposed algorithm when used in a reasonably broad class of non-convex min-max optimization problems. Experimental results illustrate its superior performance vis-a-vis benchmark methods for solving such problems.
\end{abstract}

\section{Introduction}
Stochastic first-order methods are of core practical importance for solving numerous optimization problems including training deep neural networks (DNN). Standard stochastic gradient descent (SGD) has become a widely used technique for the latter task. However, its convergence crucially depends on the tuning and update of the learning rate over iterations in order to control the variance of the gradient in the stochastic search directions, especially for non-convex functions \cite{bottou2018optimization}.

To alleviate these issues, several improved variants of SGD that automatically update the search directions and learning rates using a metric constructed from the history of iterates have been proposed, including adaptive methods \cite{jacobs1988increased,becker1988improving,duchi2011adaptive,mcmahan2010adaptive} and adaptive momentum methods \cite{kingma2014adam,reddi2018convergence}. In particular, \textsc{Adam} belonging to the second category enjoys the dual advantages of variance adaption and momentum direction \cite{nesterov1983method,polyak1964some} and hence represents a popular algorithm to train DNNs.

There is a large body of literature on the theoretical and empirical benefits of adaptive momentum optimization algorithms for convex \cite{kingma2014adam,reddi2018convergence}, smooth non-convex \cite{chen2018convergence,zaheer2018adaptive,nazari2019dadam}, and non-smooth non-convex settings \cite{nazari2020adaptive}. \cite{liu2019towards} gives an analysis of an optimistic adaptive method that uses~\textsc{Adagrad}~\cite{duchi2011adaptive,mcmahan2010adaptive} for non-convex min-max optimization. However,~\textsc{Adagrad}-type methods are suited for sparse convex settings and their performance deteriorates in (dense) non-convex optimization problems~\cite{zaheer2018adaptive}. These empirical findings necessitate the use of adaptive momentum methods that incorporate knowledge of past iterations. It is important to notice that all these methods are designed for classical minimization problems. However, training DNNs such as Generative Adversarial Networks (GANs) require solving a general class of min-max optimization problems~\cite{goodfellow2014generative, arjovsky2017wasserstein} which due to its difficulty, keeps other generative models attractive~\cite{barazandeh2019training, barazandeh2018behavior}. 

The goal of this paper is to generalize adaptive momentum methods to solve a general class of \textit{non-convex-non-concave min-max problems}. It develops an adaptive algorithm for solving min-max saddle point games and theoretically analyzes its convergence rate. The performance of the developed algorithm is assessed on training GANs.

The remainder of the paper is organized as follows. Section~\ref{sec:Problem_for} provides the formulation of the min-max problem, Section~\ref{sec:algorithm} describes the proposed algorithm and Section~\ref{sec:CA} investigates its convergence properties. Finally, Section~\ref{sec:NS} provides numerical results for training GANs.

\section{Formulation of the Min-Max Optimization Problem}\label{sec:Problem_for}

Consider the \textit{stochastic} min-max saddle point problem
\begin{equation}\label{eq:main_game}
\min\limits_{\btheta }\max\limits_{\balpha} F(\btheta, \balpha) = \E_{\bxi \sim \mathcal{D}}[f(\btheta,\balpha;\bxi)],
\end{equation}
where $\btheta \in \mathbb{R}^{p_1}$, $\balpha \in \mathbb{R}^{p_2}$, $\bxi$ is a random variable drawn from an unknown distribution $\mathcal{D}$, and $F(\btheta,\balpha)$ is a non-convex-non-concave function, i.e., it is non-convex in $\btheta$ for any given $\balpha$ and is non-concave in $\balpha$ for any given $\btheta$. 

Next, we introduce necessary notation and definitions. Throughout, $\y := (\btheta,\balpha) \in \mathbb{R}^{p_1} \times \mathbb{R}^{p_2}$ and denote the objective function of Game~\eqref{eq:main_game} and its random realization by $F(\y)$ and $f(\y;\bxi)$, respectively. Further, $ \nabla F(\y) = [\nabla_{\btheta}F(\btheta,\balpha), -\nabla_{\balpha}F(\btheta,\balpha)]$ and $\nabla f(\y;\bxi) = [\nabla_{\btheta}f(\btheta,\balpha;\bxi), -\nabla_{\balpha}f(\btheta,\balpha;\xi)]$ denotes the corresponding gradient and stochastic gradient of the objective function, respectively.

\begin{deff}[Nash Equilibrium] \label{def:ne} A point $(\btheta^*, \balpha^*) \in \mathbb{R}^{p_1} \times \mathbb{R}^{p_2}$ is a Nash equilibrium of Game~\eqref{eq:main_game} if
\begin{equation*}
F(\btheta^*, \balpha) \leq  F(\btheta^*,\balpha^*) \leq F(\btheta, \balpha^*), \ \
\forall (\btheta, \balpha) \in \mathbb{R}^{p_1} \times \mathbb{R}^{p_2}. 
\end{equation*}
\end{deff}
This definition implies that $\btheta^*$ is a global minimum of $F(\cdot, \balpha^*)$ and $\balpha^*$ is a global maximum of $F(\btheta^*, \cdot)$. In the convex-concave regime with $F(\btheta,\balpha)$ being convex in $\btheta$ for any given $\balpha$ and concave in $\balpha$ for any given $\btheta$, the Nash equilibrium always exists~\cite{jin2019minmax} and there are several algorithms for identifying it~\cite{gidel2017frank,hamedani2018iteration}. However, computing a Nash equilibrium point is NP-hard in general~\cite{daskalakis2017training,jin2019minmax}, and it may not even exist~\cite{farnia2020gans}. As a result, since we are considering the general non-convex-non-concave regime, we settle in computing a \textit{first-order Nash equilibrium} point~\cite{nouiehed2019solving,barazandeh2020solving} defined next.
\begin{deff}[First-Order Nash Equilibrium (FNE)] \label{deff:FN}
A point $\y^* \in \mathbb{R}^{p_1} \times \mathbb{R}^{p_2}$ is a first-order Nash equilibrium point of Game~\eqref{eq:main_game}, if $ \nabla F(\y^*) = 0$. 
\end{deff}

Note that at a FNE point, each player satisfies the first-order optimality condition of its own objective function when the strategy of the other player is fixed  \cite{pang2016unified,pang2011nonconvex}. In practice, iterative algorithms are used for computing a FNE for a stochastic problem. As a result, the performance of different iterative algorithms are evaluated based on the following \textit{approximate} stochastic FNE definition.
\begin{deff}[$\epsilon$-Stochastic First-Order Nash Equilibrium (SFNE)] \label{deff:EFN} A random variable $\y^*$ is an approximate SFNE~($\epsilon$-SFNE) point of Game~\eqref{eq:main_game} if 
$
\E\left[\|\nabla F(\y^*)\|^2 \right] \leq \epsilon^2,
$
where the expectation is taken over the distribution of the random variable $\y^*$. 
\end{deff}
The randomness of variable $\y^*$ in Definition 3 comes from the use of iterative algorithms that have access to stochastic gradients of the objective function (see, e.g., Algorithm \ADAM  \ below). The objective of this work is to find an $\epsilon$-SFNE point for Game~\eqref{eq:main_game} using an iterative method based on adaptive momentum.

\section{THE \ADAM Algorithm}\label{sec:algorithm}

The proposed ADAptive Momentum Min-Max (\ADAM) algorithm comes with convergence guarantees for solving a general class of \textit{non-convex-non-concave} saddle point games defined in~\eqref{eq:main_game}.  It is obtained by integrating \textsc{AMSGrad} \cite{reddi2018convergence}, a modified version of \textsc{Adam} \cite{kingma2014adam}, with a stochastic extra-gradient method \cite{iusem2017extragradient}.
\begin{algorithm}
    \SetAlgoLined
 	\SetKwInOut{Input}{Input}
	\SetKwInOut{Output}{Output}
\Input{$\{\beta_{1,k}\}_{k=1}^N, \beta_2, \beta_3 \in [0,1)$, $m \in \mathbb{N}$, and $ \eta \in \R_+$\;}
\BlankLine
Initialize $\z_{0} = \x_{0}  = \m_{0}= \wi_{0} = \textbf{d}_{0}= \bm{0}_{d}$. 
\BlankLine
		\For{$ k =1:N$}
	{
	 
	    $\z_{k} =  \x_{k-1}  - \eta \textbf{d}_{k-1}$; 
	    
	    Draw $\bxi_{k} = (\bxi_{k}^{1}, \cdots, \bxi_{k}^{m})$ from $\mathcal{D}$, and set $\seg_{k} =\frac{1}{m} \sum_{i = 1}^m \nabla f(\z_{k};\bxi_{k}^{i})$; 
	    
	    $\m_{k} = \beta_{1,k} \m_{k-1} + (1-\beta_{1,k}) \seg_{k}$; 
	    
	    $ \wi_{k} =   \beta_2 \wi_{k-1} + (1-\beta_2) \seg_{k} \odot \seg_{k}$;
	    
	    $ \sVtori_{k} = \beta_3\sVtori_{k-1} + (1-\beta_{3})\max(\sVtori_{k-1}, \wi_{k} )$; 
	    
	    $\textbf{d}_{k} =  {\sVtori_{k}^{-\frac{1}{2}}} \odot \m_{k}$; 
	    
	    $\x_{k} =  \x_{k-1} - \eta \;\textbf{d}_{k}$; 
	}
     \caption{ADAptive Momentum Min-Max~({\ADAM})}
     \label{alg:0}
     $\odot$: Element-wise vector multiplication
    \end{algorithm}
As seen in Algorithm~\ref{alg:0}, \ADAM~generates two sequences ${\x_k}$ and ${\z_k}$, where ${\x_k}$ is an ancillary sequence and the stochastic gradient is only computed over the sequence of $\z_{k}$'s using a mini-batch of size $m$, i.e., $\seg_{k} = 1/m \sum_{i = 1}^m \nabla f(\z_{k};\bxi_{k}^i)$. Using a mini-batch for estimating the gradient is a commonly used approach and more details are available in~\cite{lin2019gradient} and references therein. After estimating the gradient,  the algorithm calculates the momentum direction, $\m_{k}$, as an exponential moving average of the past gradients. Then, $\m_{k}$ is adaptively scaled by the square root of the exponential moving average of squared past gradients $\sVtori_{k}$.

The following remarks about \ADAM{} are in order: \\
\textbf{(1)}
The square and the maximum operators are applied element-wise. In some applications, to prevent division by zero, we may add a small positive constant $\epsilon$ to $\wi_{k}$ \cite{nazari2019dadam}. Further, a mini-batch of size $m$ is used in each iteration to estimate the gradient's value. \\
\textbf{(2)}
\ADAM{} computes adaptive learning rates from estimates of the second moments of the gradients, similar to \cite{nazari2019dadam}. In particular, it uses a larger learning rate compared to \textsc{AMSGrad} and yet incorporates the intuition of slowly decaying the effect of previous gradients on the learning rate. The decay parameter $\beta_{3}$ is an important component of \ADAM{}, that enables establishing its convergence properties similar to \textsc{AMSGrad} ($\beta_3= 0$), while maintaining the efficiency of \textsc{Adam}.

\section{Convergence Analysis}\label{sec:CA}
We start by positing the following assumptions:

\begin{assum}\label{assumption:function}
For all $\x \in \mathbb{R}^d$, 
\begin{enumerate}
    \item  \label{assumption:g_unbi} $\E_{\bxi \sim \mathcal{D}}[\nabla f(\x,\bxi)] = \nabla F(\x)$.
    \item  \label{assumption:g_bounded} The function $f(\x,\bxi)$ has a $G_{\infty}$-bounded gradient, i.e., $\forall \ \bxi \sim \mathcal{D}$, it holds that $\|\nabla f(\x,\bxi)\|_{\infty} \leq G_{\infty} <\infty$. 
    \item  \label{assu:bounded_var}  The function $F$ has bounded variance, i.e., 
$$\mathbb{E}_{\bxi \sim \mathcal{D}}\left[\|\nabla f(\x,\bxi) -\nabla F(\x)\|^2\right] = \sigma^2<\infty.$$
\end{enumerate}
\end{assum}
The above assumptions are fairly standard in the non-convex optimization literature \cite{bottou2018optimization,gower2019sgd}. Further, Assumption~\ref{assumption:function}(\ref{assumption:g_bounded}) is slightly stronger than the assumption $ \|\nabla f(\x,\bxi)\| \leq G_2$ that is commonly used in the analysis of stochastic gradient descent. However, Assumption~\ref{assumption:function}(\ref{assumption:g_bounded}) is crucial for the convergence analysis of adaptive methods~\cite{zaheer2018adaptive,liu2019towards,nazari2019dadam,nazari2020adaptive}.
\begin{assum}[Lipschitz Gradient]\label{assumption:g_lip}
The function $F$ is $L$-smooth, i.e., 
\begin{align*}
\|\nabla F(\x) - \nabla F(\textbf{y})\| \leq L \|\x - \textbf{y}\|, \quad 
\textnormal{for all} \quad \x, \textbf{y} \in \mathbb{R}^d.    
\end{align*}
\end{assum} 
The above assumption is standard and commonly used in the optimization literature~\cite{nesterov1998introductory,nesterov2013gradient}.
\begin{assum}[Minty condition]\label{assumption:minty} There exits $\x_*\in \mathbb{R}^d$ such that for any $\x \in \mathbb{R}^d$ we have
\[
\langle \x-\x_*, \nabla F(\x) \rangle \geq 0.
\]
\end{assum}
As explained in~\cite{razaviyayn2020nonconvex,liu2019towards} and references therein, the Minty condition is a commonly used assumption in the literature for analyzing non-convex min-max games and is weaker than other benchmark assumptions such as pseudo-monotonicity or monotonicity~\cite{mertikopoulos2019optimistic}. 
\begin{assum}\label{assumption:bounded-space} 
For the point $\x^*$ satisfying the Minty condition and all iterates $k$ generated by Algorithm~\ref{alg:0}, we have $\|\x_*\| \leq \frac{D}{2}$ and $\|\x_{k}\| \leq \frac{D}{2}$. 
\end{assum}
This assumption is required in the analysis of min-max saddle point games and has been used in~\cite{liu2019towards, liu2019decentralized}. This assumption holds true in the training process of DNNs that have normalization layers in their structure~\cite{tan2020improved, karras2017progressive,kurach2019large}.
\begin{assum}\label{assumption:bounded-v} 
In Algorithm~\ref{alg:0}, $G^2_{0} \leq \|\sVtori_0\|_{\infty}$.  
\end{assum}
This assumption is required in the analysis of adaptive methods~\cite{nazari2020adaptive,nazari2019dadam} and can be easily satisfied in the initialization step of the proposed algorithm. Next, we introduce lemmas used to establish the main result.
\begin{lemma}\label{lem:useful_bounds}[\cite{tran2019convergence}, Lemma 4.2] Let Assumption~\ref{assumption:function}~(\ref{assumption:g_bounded}) hold. Then, in Algorithm~\ref{alg:0} we have $\|\m_{k}\|_{\infty} \leq G_{\infty}$ and $\|\sVtori_{k}\|_{\infty} \leq G_{\infty}^2$ for all $k \in \{1 \cdots N \}$. 
\end{lemma}
\begin{lemma}\label{VM_bound}
Assume that $\gamma := \beta_{1,1}/\beta_2 \leq 1$ in Algorithm~\ref{alg:0}. Then, for each $k \in \{1 \cdots N \}$ we have
\begin{align*}
\|\tilde{v}_{k}^{-\frac{1}{2}} \circ m_{k-1}\| \leq \sqrt{\frac{d}{u_c}}, 
\end{align*}
where $u_c := (1-\beta_3)(1-\beta_{1,1}) (1-\beta_2) (1-\gamma). $
\end{lemma}
The following Theorem~\ref{thm:main} establishes the main result by providing an upper bound for the average norm of the gradient of the objective function.  
\begin{thm}\label{thm:main}
Let Assumptions \ref{assumption:function}--\ref{assumption:bounded-v} hold, and $L$, $G_\infty$, $G_0$, $\sigma$ be defined therein. In Algorithm~\ref{alg:0}, if we choose 
\begin{align*}
&\eta \leq \sqrt{G_0^3/(56L^2G_{\infty})},\beta_{1,k} = \beta_{1,1} \kappa^{k-1}, \beta_{1,1} \leq \frac{\sqrt{C}}{\sqrt{C} + 1}
\end{align*}
where $\kappa \in (0,1)$ and $ C = \frac{(1+\kappa)\kappa^2 G_0^3}{168(1-\kappa) G^3_{\infty}}$, then 
\begin{align}\label{eqn:upgrad}
\frac{1}{N}\sum_{k = 1}^N \E \|\nabla F (\z_k)\|^2 \leq  \frac{C_1}{N}+ \frac{ C_2\sigma^2 }{m},
\end{align}
\end{thm}
for some positive constants $C_1$ and $C_2$.

\begin{cor}\label{cor:1}
Under assumptions in Theorem~\ref{thm:main}, if $N \geq 3 C_1 \epsilon^{-2}$ and $m \geq 3 C_2 \sigma^2 \epsilon^{-2}$, then there exists an iterate $\z_k$, $k \in \{1, \cdots, N\}$ that is an $\epsilon$-SFNE point of Game~\eqref{eq:main_game}.
\end{cor}

\begin{cor}\label{cor:2}
Algorithm~\ref{alg:0} requires $\mathcal{O} (\epsilon^{-4})$ gradient evaluations of the objective function to find an $\epsilon$-SFNE point  of Game~\ref{eq:main_game}. This is consistent with other adaptive methods such as~\cite{liu2019towards}.
\end{cor}

\section{Numerical Studies}\label{sec:NS}

\textbf{(I) A Synthetic Data Experiment:} \\
Simultaneous \SADAM{} (\SSADAM) is one of most commonly used approaches for solving min-max problems that are formulated using DNNs such as training GANs~\cite{goodfellow2014generative}. In this method, the minimization and the maximization parameters are updated simultaneously using the \SADAM~algorithm~\cite{kingma2014adam}. However, this method fails drastically in solving simple min-max problems.
To better understand this issue, consider solving the following simple stochastic min-max problem    
\begin{equation}\label{eq:numerical}
  f(\theta,\alpha) =
    \begin{cases}
      c(\theta-\alpha) + (\theta^2-\alpha^2)+ k \theta \alpha,  &  \text{w.p\;} \frac{1}{3},  \\
      (\theta-\alpha) + (\theta^2-\alpha^2) + k \theta \alpha, & \text{w.p\;} \frac{2}{3},
    \end{cases}       
\end{equation}
where $c > 1$ and $k \geq 0$. Some calculations lead to 
\begin{equation*}
F(\theta,\alpha) = 
\frac{c+2}{3} (\theta-\alpha)+ (\theta^2-\alpha^2)+ k \theta \alpha.    
\end{equation*} 
This problem has the following unique FNE  
\begin{align*}
(\theta^*,\alpha^*) =-\frac{c+2}{3k^2 + 12} (2-k,2+k).    
\end{align*}
Since $\nabla^2_{\theta} F(\theta,\alpha) = 2\I \succ 0$ and $\nabla^2_{\alpha}F(\theta,\alpha) = -2\I \prec 0$, this function is strongly-convex-strongly-concave and many available algorithms ~\cite{thekumparampil2019efficient, ostrovskii2020efficient} can compute its FNE due to its special structure. 

This case study shows that despite the simplicity of the problem,~\SSADAM~ is unable to recover the single FNE point of this function. We also compare the performance of~\SSADAM~ with our proposed algorithm. 
To do the comparison, we define
$e_k = \frac{\|\z_k - \z_*\|}{\|\z_*\|}$ such that $\z_k = (\theta_k, \alpha_k)$ and $\z_* = (\theta_*, \alpha_*)$
and $\mathcal{R}_k = \frac{1}{k} \sum_{i = 1}^k \| \nabla F(\z_k)\|^2$  to measure the performance of different methods. We set the parameters at $c = 1010, k = 0.01, N = 10^{7}, \eta = 10^{-2}, \beta_1 = 0, \beta_2 = 1/(1+c^2)$ and $\beta_3 = 0.1$. All other parameters are initialized at zero. Figure~\ref{fig:twoinone} shows the result of the experiment. We have assigned 2 different scales on the vertical dimension due to space limitations. The left axis depicts the error rate, $e_k$, and the right one the average norm of the gradient, $\mathcal{R}_k$. \ADAM{} converges to the only FNE point, while~\SSADAM~is unable to locate it. This shows that \SSADAM~is unreliable even for a simple strongly-convex-strongly-concave problem.

\begin{figure}[h]
\centering
\includegraphics[scale=0.25]{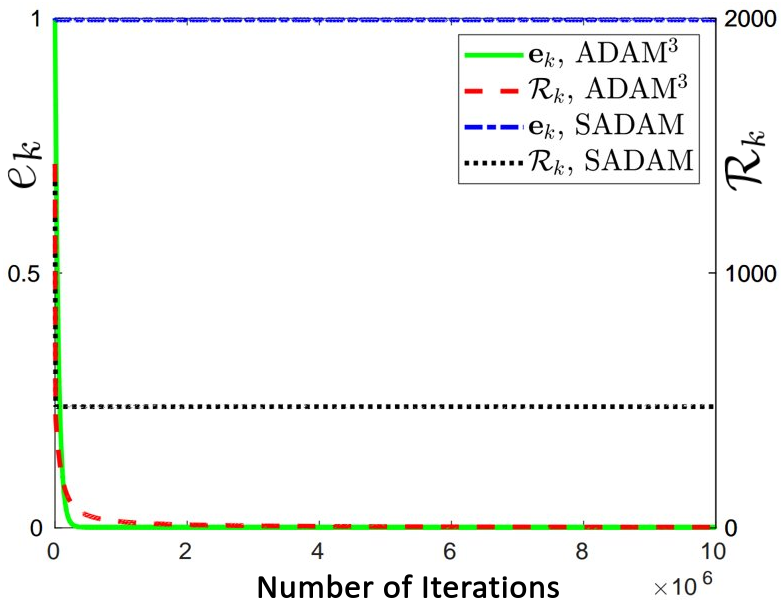}
\caption{Left/Right $y$-axis: Error rate, $e_t$ / Average norm of gradient, $\mathcal{R}_k$. \SSADAM~ misses the unique FNE point.}
\label{fig:twoinone}
\end{figure}
\textbf{(II) Training GANs with \ADAM: } \\
Algorithm~\ref{alg:0} is used
to train GANs on the publicly available CIFAR-10 data set, containing $60000$ color images of size $32 \times 32$ in $10$ different classes (see  \texttt{\url{https://www.cs.toronto.edu/~kriz/cifar.html}}).

\textit{Models and tasks:}
The generator's network consists of the input layer, 2 hidden layers and the output layer. Each of the input and hidden layers consist of a transposed convolution layer followed by batch normalization and a ReLU activation function. The output layer is a  transposed convolution layer with a hyperbolic tangent activation function. 
The network for the discriminator also has the input layer, 2 hidden layers and the output layer. Both the input and hidden layers are convolutional layers followed by instance normalization and a Leaky ReLU activation function with slope~$0.2$.
The output layer consists only of a convolutional layer.
The scripts containing the detail design of the networks, together with the implementation of \ADAM and its competitor Optimistic AdaGrad (OAdagrad) ~\cite{liu2019towards} in PyTorch, will be available at~\texttt{\url{https://github.com/babakbarazandeh}}. 

The  parameters are set to  $\eta = 0.5 \times 10^{-3}$, $\beta_1 = 0.5$, $\beta_2 = 0.9$ and $\beta_3 = 0.5$, respectively and the batch size to $64$. Finally, the experiment runs for a total of $40,000$ iterations. Figure~\ref{fig:IS}
 depicts the inception score of the generated images, a metric that evaluates their quality
 \cite{salimans2016improved}. It can be seen that \ADAM{} exhibits better performance than OAdagrad at all iteration stages. Some generated samples are available at Figure~\ref{fig:sample}.
\begin{figure}[h]
\centering 
\includegraphics[scale = 0.25]{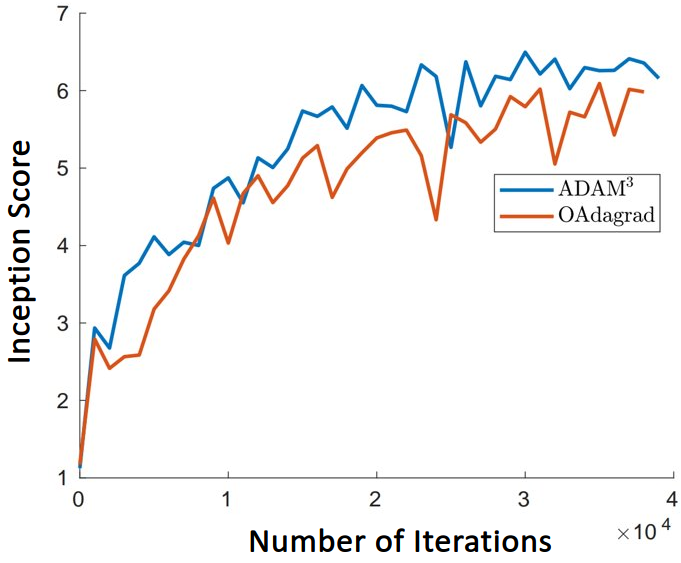}
\caption{Inception score for generated CIFAR-10 images using \ADAM and OAdagrad. }
\label{fig:IS}
\end{figure}



\section*{Acknowledgement}
The work of Babak Barazandeh was supported by the UF Informatics Institute and of George Michailidis by NSF grants DMS 1854476 and DMS 1830175. The authors would also like to thank Dr. Meisam Razaviyayn for his insightful comments that helped to improve the quality of the work.

\bibliographystyle{IEEEbib}

\bibliography{Icassp_single_column}
\normalsize
\onecolumn
\section*{Appendix}
In this section, we provide proofs for Theorem~\ref{thm:main} and required auxiliary lemmas.
\begin{rmk}\label{lem:hadamard} 
For any set of vectors $\{\textbf{a}_i\}_{i = 1}^M, \textbf{b}$ and $\textbf{c}$ in $\mathbb{R}^d$ we have 
$$1.\|\sum\limits_{i = 1}^M \textbf{a}_{i}\|^2\leq M \sum\limits_{i = 1}^M \|\textbf{a}_i\|^2, \qquad 2.\|\textbf{a} \circ \textbf{b}\| \leq \|\textbf{a}\|_{\infty} \|\textbf{b}\|_{1}, \qquad 3. \|\textbf{b} \circ \textbf{c}\| \leq \|\textbf{b}\|_{\infty} \|\textbf{c}\|.$$
\end{rmk}

\textbf{Lemma~\ref{VM_bound}.}~\textit{Assume that $\gamma := \beta_{1,1}/\beta_2 \leq 1$ in Algorithm~\ref{alg:0}. Then, for each $k \in \{1 \cdots N \}$ we have
\begin{align*}
\|\tilde{v}_{k}^{-\frac{1}{2}} \circ m_{k-1}\| \leq \sqrt{\frac{d}{u_c}}, 
\end{align*}
where $u_c := (1-\beta_3)(1-\beta_{1,1}) (1-\beta_2) (1-\gamma). $}
\begin{proof}
For each $k \in \{1 \cdots N \}$ and $r \in \{1, \cdots, d\}$, let $\tilde{v}_{r, k}^{-\frac{1}{2}}$ and $m_{r,k}$ represent the values of the $r^{th}$ 
coordinate of vectors $\sVtori_{k}^{-\frac{1}{2}}$ and $\m_{k}$, respectively. Then, from the update rule of Algorithm~\ref{alg:0} we have
$$\tilde{v}_{r,k} = \beta_3 \tilde{v}_{r,k-1} + (1-\beta_{3})\max(\tilde{v}_{r,k-1}, v_{r,k}),$$ 
which implies that $ \tilde{v}_{r,k} \geq (1-\beta_{3}) v_{r,k}$. Besides, it can be easily seen from the update rule of $\m_{k}$ and $ \wi_{k}$ in Algorithm~\ref{alg:0} that 
\begin{align*}
     m_{r,k} = \sum_{s = 1}^k \left( \prod\limits_{l = s+1}^k \beta_{1,l} \right)  (1-\beta_{1,s}) \hat{g}_{r,s}, \quad \text{and} \quad  v_{r,k}   = (1-\beta_{2})\sum_{s = 1}^k \beta_{2}^{k-s}\hat{g}_{r,s}^2. 
\end{align*}
Thus,
\begin{align}\label{eqn:bvb}
\nonumber
|v_{r,k}^{-\frac{1}{2}} m_{r,k-1}|^2 \leq |v_{r,k-1}^{-\frac{1}{2}} m_{r,k-1}|^2 &\leq    \frac{\left(\sum\limits_{s = 1}^{k-1} \left( \prod\limits_{l = s+1}^{k-1} \beta_{1,l} \right)  (1-\beta_{1,s}) \hat{g}_{r,s}\right)^2 }{(1-\beta_2)\sum\limits_{s = 1}^{k-1} \beta_{2}^{k-s-1} \hat{g}_{r,s}^2} 
\\
& \leq   \frac{\left(\sum\limits_{s = 1}^{k-1} \left( \prod\limits_{l = s+1}^{k-1} \beta_{1,l} \right)  \hat{g}_{r,s}\right)^2 }{(1-\beta_2)\sum\limits_{s = 1}^{k-1} \beta_{2}^{k-s-1} \hat{g}_{r,s}^2},
\end{align}
where the first inequality follows since ${v}_{r,k}^{-\frac{1}{2}} \leq {v}_{r,k-1}^{-\frac{1}{2}}$ for all $r \in [d]$ and the last inequality uses  our assumption that $\beta_{1,s} \leq 1$ for all $s \geq 1$.

Now, let $\pi_s =  \prod\limits_{l = s+1}^{k-1} \beta_{1,l}$. Since $\beta_{1,l}$ is decreasing, we get $\pi_s \leq \beta_{1,1}^{k-s-1}$. This, together with $(\sum_{i} a_i b_i)^2 \leq (\sum_{i} a_i^2) (\sum_{i} b_i^2)$ implies that  
\begin{align*}
\frac{\left(\sum\limits_{s = 1}^{k-1}  \pi_s  \hat{g}_{r,s}\right)^2 }{(1-\beta_2)\sum\limits_{s = 1}^{k-1} \beta_{2}^{k-s-1}\hat{g}_{r,s}^2} & \leq \frac{(\sum\limits_{s = 1}^{k-1} \pi_s)(\sum\limits_{s = 1}^{k-1} \pi_s  \hat{g}_{r,s}^2 )}{(1-\beta_2)\sum\limits_{s = 1}^{k-1} \beta_{2}^{k-s-1}\hat{g}_{r,s}^2}
\\
&\leq \frac{1}{1-\beta_2} (\sum_{s = 1}^{k-1}  \pi_s)\left(\sum_{s = 1}^{k-1} \frac{\pi_s  \hat{g}_{r,s}^2}{\beta_{2}^{k-s-1}\hat{g}_{r,s}^2}\right) \\
& \leq \frac{1}{1-\beta_2}  (\sum_{s = 1}^{k-1}  \pi_s)\sum_{s = 1}^{k-1} \frac{\pi_s}{\beta_{2}^{k-s-1}} 
\\
&\leq \frac{1}{1-\beta_2}\frac{1}{1-\beta_{1,1}} \frac{1}{1- \gamma}, 
\end{align*}
where the last inequality follows from our assumption $\gamma= \frac{\beta_{1,1}}{\beta_2} \leq 1$. Finally, substituting the above inequality into \eqref{eqn:bvb} yields the desired result.
\end{proof}

\begin{lemma}\label{lem:v_diff}
For each $k \in \{1 \cdots N \}$ and $r \in \{1, \cdots, d\}$, let $\tilde{v}_{r, k}$ represent the value of the $r^{th}$ coordinate of vector $\sVtori_{k}$. Then, for the sequence of $\sVtori_k$'s generated by Algorithm~\ref{alg:0} we have
\begin{enumerate}
\item \label{itm:one_Vp} $\sum\limits_{k=1}^{N} \|\sVtori_{k}^{p} - \sVtori_{k-1}^{p} \|_{1} \leq \sum_{r=1}^d \max \left( \tilde{v}_{r,0}^{p} , \tilde{v}_{r,N}^{p} \right)$ \quad  \text{and}
\item \label{itm:two_Vp} $ \sum\limits_{k=1}^{N} \|\sVtori_{k}^{p} - \sVtori_{k-1}^{p} \|_{1}^2 \leq \sum_{r=1}^d  \tilde{v}_{r,0}^{p} \max \left(\tilde{v}_{r,0}^{p} , \tilde{v}_{r,N}^{p} \right)$
\end{enumerate}
where  $p \in \mathbb{R}$ and the vector powers are considered to be element-wise.
\end{lemma}
\begin{proof}
\ref{itm:one_Vp}.~If $p>0$, from the update rule of $\sVtori_{k}$ in Algorithm~\ref{alg:0} we have 
\begin{align*}
\nonumber
\sum_{k=1}^{N}  \left\Vert \sVtori_{k}^{p} - \sVtori_{k-1}^{p}\right\Vert_1  =  \sum_{k=1}^{N}  \sum_{r=1}^d  (\tilde{v}_{r,k}^{p} - \tilde{v}_{r,k-1}^{p})  & = \sum_{r=1}^d \sum_{k=1}^{N}  (\tilde{v}_{r,k}^{p} - \tilde{v}_{r, k-1}^{p}) 
\\
& \leq   \sum_{r=1}^{d}  \tilde{v}_{r,N}^{p} , 
\end{align*}
where the first equality is due to the fact that for $p >0 $, each element of $\sVtori_k^p$ is increasing in $k$ and the last inequality uses the telescoping sum. Now, we consider the case when $p<0$. It can be easily seen that 
\begin{align*}
\nonumber
 \sum_{k=1}^{N} \left\Vert \sVtori_{k}^{p} - \sVtori_{k-1}^{p}\right\Vert_1  &= \sum_{k=1}^{N}  \sum_{r=1}^d  (-\tilde{v}_{r,k}^{q} + \tilde{v}_{r,k-1}^{q}) 
 \leq   \sum_{r=1}^{d}  \tilde{v}_{r,0}^{p}.
\end{align*}
\\
\ref{itm:two_Vp}.~For $p>0$, it follows that 
\begin{align*}
\nonumber
\sum_{k=1}^{N}  \left\Vert \sVtori_{i,k}^{p} - \sVtori_{i,k-1}^{p}\right\Vert_1^2 \leq 
 \sum_{k=1}^{N}   \sum_{r=1}^d \left(\tilde{v}_{r,k}^{p} - \tilde{v}_{r,k-1}^{p}\right) \tilde{v}_{r,k}^{p}  & \leq   \sum_{k=1}^{N}  \sum_{r=1}^d \left(\tilde{v}_{r,k}^{p} - \tilde{v}_{r,k-1}^{p}\right) \tilde{v}_{r,N}^{p} \\
  & \leq \sum_{r=1}^{d}  \left(\tilde{v}_{r,0}^{p} - \tilde{v}_{r,N}^{p}\right)  \tilde{v}_{r,N}^{p}  \\
& \leq \sum_{r=1}^{d}  \tilde{v}_{r,0}^{p} \tilde{v}_{r,N}^{p}.  
\end{align*}
Now, we consider the case when $p<0$. It can be easily seen that 
\begin{align*}
\nonumber 
\sum_{k=1}^{N}  \left\Vert \sVtori_{k}^{p} - \sVtori_{k-1}^{p}\right\Vert_1^2  \leq 
 \sum_{k=1}^{N}   \sum_{r=1}^d (-\tilde{v}_{r,k}^{p} + \tilde{v}_{r,k-1}^{p}) (\tilde{v}_{r,k-1}^{p}) & \leq   \sum_{k=1}^{N}  \sum_{r=1}^d (-\tilde{v}_{r,k}^{p} + \tilde{v}_{r,k-1}^{p}) \tilde{v}_{r,0}^{p} \\
& \leq    \sum_{r=1}^{d}  \tilde{v}_{r,0}^{p} \tilde{v}_{r,0}^{p}.  
\end{align*}
\end{proof}

\begin{lemma}\label{lem:teles}
Under Assumptions~\ref{assumption:function} and~\ref{assumption:bounded-space} we have
\begin{align*}
\sum_{k = 1}^N \left(\left\|\sVtori_{k-1}^{\frac{1}{4}} \circ (\x_{k-1}-\x_*)\right\|^2 - \left\|\sVtori_{k-1}^{\frac{1}{4}} \circ (\x_{k}-\x_*)\right\|^2\right) \leq 3 D^2 dG_{\infty}.
\end{align*}
\end{lemma}
\begin{proof}
Observe that 
\begin{align}\label{sumd2_0}
\nonumber
& \quad \sum_{k = 1}^N \left(\left\|\sVtori_{k-1}^{\frac{1}{4}} \circ (\x_{k-1}-\x_*)\right\|^2 - \left\|\sVtori_{k-1}^{\frac{1}{4}} \circ (\x_{k}-\x_*)\right\|^2\right)
\\ \nonumber
& =  \left\|\sVtori_{0}^{\frac{1}{4}} \circ (\x_{0}-\x_*)\right\|^2 +  \left( -\left\|\sVtori_{0}^{\frac{1}{4}} \circ (\x_{1}-\x_*)\right\|^2 + \left\|\sVtori_{1}^{\frac{1}{4}} \circ (\x_{1}-\x_*)\right\|^2\right)
\\ \nonumber
& + \left( -\left\|\sVtori_{1}^{\frac{1}{4}} \circ (\x_{2}-\x_*)\right\|^2 + \left\|\sVtori_{2}^{\frac{1}{4}} \circ (\x_{2}-\x_*)\right\|^2\right)
\\ \nonumber
& \qquad \vdots
\\ 
& + \left( -\left\|\sVtori_{N-2}^{\frac{1}{4}} \circ (\x_{N-1}-\x_*)\right\|^2 + \left\|\sVtori_{N-1}^{\frac{1}{4}} \circ (\x_{N-1}-\x_*)\right\|^2\right)  -\left\|\sVtori_{N-1}^{\frac{1}{4}} \circ (\x_{N}-\x_*)\right\|^2.
\end{align}
For arbitrary $s^{th}$ pairs in~\eqref{sumd2_0}, we have
\begin{align}\label{eq:sum_temp}
\nonumber
& -\left\|\sVtori_{s-1}^{\frac{1}{4}} \circ (\x_{s}-\x_*)\right\|^2 + \left\|\sVtori_{s}^{\frac{1}{4}} \circ (\x_{s}-\x_*)\right\|^2 
\\ \nonumber
&=  \left(-\left\|(\sVtori_{s-1}^{\frac{1}{4}}-\sVtori_{s}^{\frac{1}{4}} + \sVtori_{s}^{\frac{1}{4}}) \circ (\x_{s}-\x_*)\right\| + \left\|\sVtori_{s}^{\frac{1}{4}} \circ (\x_{s}-\x_*)\right\|\right) \\
\nonumber
& \cdot \left(\left\|\sVtori_{s-1}^{\frac{1}{4}} \circ (\x_{s}-\x_*)\right\| + \left\|\sVtori_{s}^{\frac{1}{4}} \circ (\x_{s}-\x_*)\right\|\right) 
\\ \nonumber
& \leq   \left(\left\|(\sVtori_{s-1}^{\frac{1}{4}}-\sVtori_{s}^{\frac{1}{4}} ) \circ (\x_{s}-\x_*)\right\| \right) 2\sqrt{G_{\infty}} D
\\
& \leq 2D^2\sqrt{G_{\infty}}\| \sVtori_{s-1}^{\frac{1}{4}}-\sVtori_{s}^{\frac{1}{4}}\|_1,
\end{align}
where the first inequality follows from $\|\textbf{a}\|-\|\textbf{b}\|\leq \|\textbf{a}-\textbf{b}\|$, Assumption~\ref{assumption:bounded-space} and Lemma~\ref{lem:useful_bounds}.

As a result, 
\begin{align}
\nonumber
\sum_{k = 1}^N \left(\left\|\sVtori_{k-1}^{\frac{1}{4}} \circ (\x_{k-1}-\x_*)\right\|^2 - \left\|\sVtori_{k-1}^{\frac{1}{4}} \circ (\x_{k}-\x_*)\right\|^2\right) & \leq 2D^2 dG_{\infty}+\left\|\sVtori_{0}^{\frac{1}{4}} \circ (\x_{0}-\x_*)\right\|^2 -\left\|\sVtori_{N-1}^{\frac{1}{4}} \circ (\x_{N}-\x_*)\right\|^2
\\ \nonumber
& \leq 3 D^2 dG_{\infty},
\end{align}
where the first inequality follows from Lemma~\ref{lem:v_diff} and last inequality uses the same lemma, Assumption~\ref{assumption:bounded-space} and the fact that $d \geq 1$. 
\end{proof}
\textbf{Theorem~\ref{thm:main}}.~\textit{Let Assumptions \ref{assumption:function}--\ref{assumption:bounded-v} hold, and $L$, $G_\infty$, $G_0$, $\sigma$ be defined therein. In Algorithm~\ref{alg:0}, if we choose 
\begin{align*}
&\eta \leq \sqrt{G_0^3/(56L^2G_{\infty})}, \quad \textnormal{and} \quad \beta_{1,1} \leq \frac{\sqrt{C}}{\sqrt{C} + 1},
\end{align*}
where $ C = \frac{(1+\kappa)\kappa^2 G_0^3}{168(1-\kappa) G^3_{\infty}}$, then 
\begin{align}\label{eqn:upgrad}
\frac{1}{N}\sum_{k = 1}^N \E \|\nabla F (\z_k)\|^2 \leq  \frac{C_1}{N}+ \frac{ C_2\sigma^2 }{m},
\end{align}
for some positive constants $C_1$ and $C_2$.}
\begin{proof}
We divide the proof into four steps. In Step 1, we show that the gradient norm is bounded by the norm of search direction and auxiliary variables $\x_k$ and $\z_k$. Then in Steps~2 and~3, we give upper bounds for these terms. Finally, in Step~4, we provide the convergence analysis.
\begin{enumerate}[leftmargin=*]
\item[\textbf{Step 1}] shows that under Assumption~\ref{assumption:function}~(\ref{assumption:g_bounded}), we have
 \begin{align}\label{eq:thm_main0}
\frac{1}{N} \sum_{k=1}^{N}  \|\srg_k\|^2 \leq \frac{3}{ N\eta^2 (1-\beta_{1,1})^2 G_{\infty}^{-2}}  \sum_{k=1}^{N} \eta^2 R_{1,k} + R_{2,k},
\end{align}
where 
\begin{align}\label{eq:r120}
\nonumber 
R_{1,k} &:= \left\| -\sdi_k+(1-\beta_{1,k}) \sVtorip_{k} \circ \srg_k\right\|^2 \\
R_{2,k} &:= \left\|\z_k-\x_k\right\|^2+\left\|\z_k-\x_{k-1}\right\|^2.
\end{align}
It follows from the update rule of $\x_k$ in Algorithm~\ref{alg:0} that 
\begin{align*}
\eta (1-\beta_{1,k})\left(\sVtorip_{k} \circ \srg_k\right)&=\z_k-\x_k+\left(\x_{k-1}-\eta \sdi_k \right)\\
& - \z_k+\eta (1-\beta_{1,k}) \left(\sVtorip_k \circ \srg_k\right). 
\end{align*}
Now, using Remark~\ref{lem:hadamard}, we get 
\begin{align}
\nonumber \eta^2 (1-\beta_{1,k})^2 \left\|\sVtorip_{k} \circ \srg_k \right\|^2 &\leq 3\eta^2\left\| -\sdi_k+(1-\beta_{1,k}) \sVtorip_{k} \circ \srg_k\right\|^2\\
&+ 3 \left(\left\|\z_k-\x_k\right\|^2+\left\|\z_k-\x_{k-1}\right\|^2\right).
\end{align}
From Lemma~\ref{lem:useful_bounds}, we have $\|\sVtori_{k}^{-\frac{1}{2}}\|_{\infty} \geq G_{\infty}^{-1}$ which implies that
\begin{align*}
\left\|\sVtori_{k}^{-\frac{1}{2}} \circ \srg_k\right\|^2  \geq G_{\infty}^{-2} \|\srg_k\|^2.  
\end{align*}
Now, it follows from the above inequality and \eqref{eq:r120} that 
 \begin{align}
 \nonumber
 \eta^2 (1-\beta_{1,1})^2 G_{\infty}^{-2}\|\srg_k\|^2 & \leq  \eta^2 (1-\beta_{1,k})^2 \left\|\sVtori_{k}^{-\frac{1}{2}} \circ \srg_k\right\|^2 \leq 3 \eta^2 R_{1,k}  + 3R_{2,k},
 \end{align}
which gives \eqref{eq:thm_main0}.

\item[\textbf{Step 2}] establishes an upper bound for $R_{1,k}$ defined in \eqref{eq:r120}. More specifically, we show that 
\begin{align}\label{R1k0}
 \frac{1}{N} \sum_{k=1}^{N} R_{1,k}  &\leq  \frac{2d\beta_{1,1}^2}{Nu_c(1-\kappa^2)} +  \frac{2}{NG_0^2}\sum_{k = 1}^N \|\beps_k\|^2,
\end{align}
where $u_c$ is defined in Lemma~\ref{VM_bound} and $\beps_k = \seg_k - \srg_k$.

From the definition of $\sdi_{k}$ in Algorithm~\ref{alg:0}, we have 
\begin{align}\label{eq:update_D0}
\sdi_{k} & = \beta_{1,k} \sVtorip_{k} \circ \m_{k-1} + (1-\beta_{1,k})\sVtorip_{k} \circ \seg_{k}.
\end{align}
Hence,
$$
-\sdi_k+(1-\beta_{1,k}) \sVtorip_{k} \circ \srg_k =-\beta_{1,k} \sVtorip_{k} \circ \m_{k-1} + (1-\beta_{1,k})\sVtorip_{k} \circ \left(\srg_k-\seg_k\right),$$
which implies that 
\begin{align}\label{R_1_k0}
\nonumber
R_{1,k} & = \left\|-\beta_{1,k} \sVtorip_{k} \circ \m_{k-1} + (1-\beta_{1,k}) \sVtorip_{k}\circ \left( \srg_k-(\srg_k + \beps_k) \right) \right\|^2
\\
&\leq 2 \beta^2_{1,k} \left\|\sVtorip_{k} \circ \m_{k-1}\right\|^2+2 (1-\beta_{1,k})^2 \left\|\sVtorip_k \circ \beps_k \right\|^2.
\end{align}
Here, the equality is obtained since $\beps_k = \seg_k - \srg_k$, and the inequality follows from Remark~\ref{lem:hadamard}. 

For the first term on the R.H.S. of \eqref{R_1_k0}, it follows from Lemma~\ref{VM_bound} that
\begin{subequations}
\begin{align}\label{rd11k0}
&\left\|\sVtorip_{k} \circ \m_{k-1}\right\|^2 \leq \frac{d}{u_c}.
\end{align}
Further, for the second term on the R.H.S. of \eqref{R_1_k0}, we have
\begin{align}\label{rd12k0}
&\left\|\sVtorip_k \circ \beps_k \right\|^2 \leq  \|\sVtorip_k\|_{\infty}^2  \|\beps_{k}\|^2 \leq \|\sVtorip_0\|_{\infty}^2  \|\beps_{k}\|^2 \leq  \frac{1}{ G_0^2} \|\beps_{k}\|^2,
\end{align}
\end{subequations} 
where the inequality uses Remark~\ref{lem:hadamard}, the fact that each element of $\sVtorip_k$ is decreasing in $k$ and our assumption that $\|\tilde{\wi}_{0}^{-\frac{1}{2}}\|_{\infty} \leq  1/G_{0}$.

Substituting \eqref{rd11k0}--\eqref{rd12k0} into \eqref{R_1_k0}, we obtain
\begin{align*}
R_{1,k} \leq \frac{2d\beta_{1,k}^2}{u_c} +  \frac{2}{G_0^2}  \|\beps_{k}\|^2.
\end{align*}
Summing the above inequality over $k=1, \cdots, N$ and using the fact that $\sum_{k=1}^N  \beta_{1,k}^2 \leq \beta_{1,1}^2/ (1-\kappa^2)$, we obtain the desired result.
\item[\textbf{Step 3}] provides an upper bound for $R_{2,k}$ defined in \eqref{eq:r120}. In particular, we show that for $ \eta \leq \sqrt{G_0^3/(56L^2G_{\infty})}$, the following holds   
\begin{align}\label{eq:stp2:result}
\nonumber
  \frac{1}{N} \sum_{k=1}^{N}   R_{2,k}&\leq
 \frac{6 D^2 dG_{\infty}}{NG_0}+ \frac{4 \eta D}{NG_0} \bigg(\frac{\beta_{1,1} G_{\infty}}{1-\kappa} \sqrt{\frac{d}{u_c}}+ \frac{G^2_{\infty}d}{G_0}\bigg) + \frac{56 \eta^2 d \beta_{1,1}^2 G_{\infty}}{Nu_c(1-\kappa^2)G_{0}} \\
 &+  \frac{28 \eta^2d G_{\infty}^3}{N G_0^3} 
 + \frac{28\eta^2G_{\infty}}{N G_0^3}  \sum_{k=1}^{N} (\beta_{1,k} - \beta_{1,k-1})^2 \|\srg_k\|^2 + \frac{56 \eta^2 G_{\infty}}{G_{0}^3N}\sum_{k = 1}^N \|\beps_{k}\|^2. 
\end{align}

Let $$\sVtori_{k-1}^{\frac{1}{4}}:=\left[\hat{v}_{1,k-1}^{\frac{1}{4}}, \hat{v}_{2,k-1}^{\frac{1}{4}}, \cdots, \hat{v}_{d,k-1}^{\frac{1}{4}}\right]^\intercal.
$$ 
The update rule of~$\x_{k}$ in Algorithm~\ref{alg:0} implies that 
\begin{align*}
\nonumber & \quad  \left\|\sVtori_{k-1}^{\frac{1}{4}} \circ (\x_k-\x)\right\|^2 = \left\Vert \sVtori_{k-1}^{\frac{1}{4}} \circ (\x_{k-1} - \eta \sdi_k - \x)\right\Vert^2
\\
& = \left\Vert \sVtori_{k-1}^{\frac{1}{4}} \circ \left( \x_{k-1}- \x\right) - \eta \sVtori_{k-1}^{\frac{1}{4}} \circ \sdi_k \right\Vert^2 - \left\Vert \sVtori_{k-1}^{\frac{1}{4}} \circ \left(\x_{k-1} - \x_{k}\right) - \eta \sVtori_{k-1}^{\frac{1}{4}} \circ \sdi_k  \right\Vert^2
\\
& = \left\|\sVtori_{k-1}^{\frac{1}{4}} \circ (\x_{k-1}-\x)\right\|^2-\left\|\sVtori_{k-1}^{\frac{1}{4}} \circ (\x_{k-1}-\x_k)\right\|^2  
\\
&-2\left\langle \sVtori_{k-1}^{\frac{1}{4}} \circ (\x_{k-1}-\x), \eta  \sVtori_{k-1}^{\frac{1}{4}} \circ \sdi_k \right\rangle + 2\left\langle \sVtori_{k-1}^{\frac{1}{4}} \circ (\x_{k-1} - \x_k), \eta  \sVtori_{k-1}^{\frac{1}{4}} \circ \sdi_k\right\rangle  \\
&-2\left\langle \sVtori_{k-1}^{\frac{1}{4}} \circ \z_k, \eta  \sVtori_{k-1}^{\frac{1}{4}} \circ \sdi_k\right\rangle + 2\left\langle \sVtori_{k-1}^{\frac{1}{4}} \circ \z_k, \eta  \sVtori_{k-1}^{\frac{1}{4}} \circ \sdi_k\right\rangle
\\
& = \left\|\sVtori_{k-1}^{\frac{1}{4}} \circ (\x_{k-1}-\x)\right\|^2-\left\|\sVtori_{k-1}^{\frac{1}{4}} \circ (\x_{k-1} -\z_k+ \z_k-\x_k)\right\|^2 \\
& -2\left\langle \sVtori_{k-1}^{\frac{1}{4}} \circ (\z_k-\x), \eta  \sVtori_{k-1}^{\frac{1}{4}} \circ \sdi_k\right\rangle + 2\left\langle  \sVtori_{k-1}^{\frac{1}{4}} \circ (\z_k- \x_k), \eta  \sVtori_{k-1}^{\frac{1}{4}} \circ \sdi_k\right\rangle
\\
& = \left\|\sVtori_{k-1}^{\frac{1}{4}} \circ (\x_{k-1}-\x)\right\|^2-\left\|\sVtori_{k-1}^{\frac{1}{4}} \circ (\x_{k-1}-\z_k)\right\|^2-\left\|\sVtori_{k-1}^{\frac{1}{4}} \circ (\z_k-\x_k)\right\|^2 \\
& + 2\left\langle \sVtori_{k-1}^{\frac{1}{4}} \circ (\x-\z_k), \eta  \sVtori_{k-1}^{\frac{1}{4}} \circ \sdi_k \right\rangle  + 2\left\langle \sVtori_{k-1}^{\frac{1}{4}} \circ (\x_k-\z_k),\sVtori_{k-1}^{\frac{1}{4}} \circ (\x_{k-1} -\z_k)\right\rangle 
\\ & +  2\left\langle \sVtori_{k-1}^{\frac{1}{4}} \circ (\x_k-\z_k),-\eta \sVtori_{k-1}^{\frac{1}{4}} \circ \sdi_k \right\rangle,\end{align*}
where the second equality follows since $\x_{k-1} - \x_{k} - \eta \sdi_k=0$.
 
Now, substituting $\x =\x_*$ into the above equality and rearranging the terms, we get  
\begin{align} \label{eq:R2k00}
\nonumber  & \quad \left\|\sVtori_{k-1}^{\frac{1}{4}} \circ (\z_k-\x_k)\right\|^2 + \left\|\sVtori_{k-1}^{\frac{1}{4}} \circ (\x_{k-1}-\z_k)\right\|^2  \\
\nonumber
&= \underbrace{ \left\|\sVtori_{k-1}^{\frac{1}{4}} \circ (\x_{k-1}-\x_*)\right\|^2 - \left\|\sVtori_{k-1}^{\frac{1}{4}} \circ (\x_{k}-\x_*)\right\|^2}_{ R_{2,0,k} }\\
\nonumber
&+2 \eta \underbrace{\left\langle \sVtori_{k-1}^{\frac{1}{4}} \circ (\x_*-\z_k), \sVtori_{k-1}^{\frac{1}{4}} \circ \sdi_k \right\rangle}_{R_{2,1,k}}\\
&+2 \eta \underbrace{\left\langle \sVtori_{k-1}^{\frac{1}{4}} \circ (\x_k-\z_k),  \sVtori_{k-1}^{\frac{1}{4}} \circ  (\sdi_{k-1} - \sdi_k)\right\rangle}_{R_{2,2,k}}.   
\end{align}
Since by our assumption $ \|\sVtori_{0}^{\frac{1}{2}}\|_{\infty} \geq G_0$, we have 
\begin{align*}
\nonumber G_0 \|\z_k - \x_k\|^2 &\leq \left\|\sVtori_{k-1}^{\frac{1}{4}} \circ (\z_k-\x_k)\right\|^2, \quad \text{and} \\
 G_0 \|\x_{k-1}-\z_k\|^2 &\leq \left\|\sVtori_{k-1}^{\frac{1}{4}} \circ (\x_{k-1}-\z_k)\right\|^2.
\end{align*}
We substitute the above lower bounds into \eqref{eq:R2k00} to get
\begin{align} \label{eq:R2k_0}
\left\|\z_{k} - \x_{k} \right\|^2 +  \left\|\x_{k-1} -  \z_{k} \right\|^2  & \leq  \frac{R_{2,0,k}}{G_0} + \frac{2 \eta}{G_0}\left(R
_{2,1,k} + R_{2,2,k}\right).   
\end{align}
Next, we provide upper bounds for the terms $R_{2,1,k}$ and $R_{2,2,k}$.
\vspace{0.5cm}
\\
\textbf{Bounding $R_{2,1,k}$.} It follows from the update rule of $\sdi_k$ in~\eqref{eq:update_D0} that 
\begin{align}\label{eq:ddd0}
\nonumber
\sdi_k &= \sdi_k - (1-\beta_{1,k})\sVtorip_{k-1}\circ \seg_{k}  +(1-\beta_{1,k})\sVtorip_{k-1}\circ \seg_{k}
\\
\nonumber
& = \beta_{1,k}\sVtorip_k\circ \m_{k-1} + (1-\beta_{1,k}) (\sVtorip_k-\sVtorip_{k-1})\circ \seg_{k} \\
&+ (1-\beta_{1,k})\sVtorip_{k-1}\circ \srg_k + (1-\beta_{1,k}) \sVtorip_{k-1}\circ (\seg_{k}- \srg_{k}). 
\end{align}
To find an upper bound for $R_{2,1,k}$, we first multiply each term in \eqref{eq:ddd0} by $\sVtori_{k-1}^{\frac{1}{4}}$ and then provide an upper bound for its inner product with $\sVtori_{k-1}^{\frac{1}{4}} \circ (\x_*-\z_k)$. From Lemmas~\ref{lem:useful_bounds},~\ref{VM_bound} and Assumption~\ref{assumption:bounded-space}, we get 
\begin{subequations}
\begin{align}\label{eq:bbb10}
\left\langle \sVtori_{k-1}^{\frac{1}{4}}\circ (\x_*-\z_k), \sVtori_{k-1}^{\frac{1}{4}}\circ \sVtorip_k\circ \m_{k-1} \right\rangle  & \leq  D G_{\infty}  \left\|\sVtorip_k\circ \m_{k-1}\right\| \leq D G_{\infty} \sqrt{du_c^{-1}}.
\end{align} 
Further,
\begin{align}\label{eq:bbb30}
\nonumber
\left\langle \sVtori_{k-1}^{\frac{1}{4}}\circ (\x_*-\z_k),\sVtori_{k-1}^{\frac{1}{4}}\circ (\sVtorip_k-\sVtorip_{k-1})\circ \seg_k \right\rangle & \leq  D G_{\infty}\| (\sVtorip_k-\sVtorip_{k-1})\circ \seg_k \| \\
\nonumber
&\leq   D G_{\infty} \|\seg_k\|_{\infty}  \|\sVtorip_k-\sVtorip_{k-1} \|_{1} \\
&\leq   D G_{\infty}^2  \|\sVtorip_k-\sVtorip_{k-1}\|_{1}, 
\end{align}
where the second inequality is obtained from Remark~\ref{lem:hadamard} and the last inequality is due to Assumption~\ref{assumption:function}~(\ref{assumption:g_bounded}). From Assumption~\ref{assumption:minty}, we have 
\begin{align}\label{eq:bbb00}
\left\langle \sVtori_{k-1}^{\frac{1}{4}}\circ (\x_*-\z_k), \sVtori_{k-1}^{-\frac{1}{4}}\circ \srg_k \right\rangle  
= \left\langle \x_*-\z_{k}, \srg_{k}\right\rangle \leq 0. 
\end{align}
Further,
\begin{align}\label{eq:aaa0}
\left\langle \sVtori_{k-1}^{\frac{1}{4}} \circ(\x_*-\z_k), \sVtori_{k-1}^{-\frac{1}{4}}\circ (\seg_{k}- \srg_{k})\right\rangle &= \left\langle\x_*-\z_k,  \seg_{k}- \srg_{k}\right\rangle =: \Theta_k .
\end{align}
\end{subequations}

Now, using~\eqref{eq:aaa0}--\eqref{eq:bbb30}, we obtain
\begin{align}\label{eq:N_update0}
 R_{2,1,k} & \leq \beta_{1,k}  D G_{\infty}\sqrt{du_c^{-1}}  +  D G_{\infty}^2  \|\sVtorip_k-\sVtorip_{k-1} \|_{1} +\Theta_k.
\end{align}

\textbf{Bounding} $R_{2,2,k}$
From the update rule of $\sdi_k$ in~\eqref{eq:update_D0}, we get 
\begin{align}\label{eq:ddif0}
\nonumber
\sdi_{k}-\sdi_{k-1} &= \beta_{1,k}  \sVtorip_k\circ \m_{k-1} +  (1-\beta_{1,k}) \sVtorip_k \circ \seg_k\\
\nonumber
    &-\beta_{1,k-1}  \sVtorip_{k-1}\circ \m_{k-2} - (1-\beta_{1,k-1}) \sVtorip_{k-1} \circ \seg_{k-1}\\
    \nonumber
    & =\beta_{1,k} \sVtorip_k\circ \m_{k-1}-\beta_{1,k-1}\sVtorip_{k-1}\circ \m_{k-2}\\
    \nonumber
    & + (1-\beta_{1,k}) (\sVtorip_k - \sVtorip_{k-1} + \sVtorip_{k-1})\circ \seg_{k}-(1-\beta_{1,k-1})\sVtorip_{k-1} \circ \seg_{k-1}\\
    \nonumber
    & = \beta_{1,k} \sVtorip_k\circ \m_{k-1} - \beta_{1,k-1} \sVtorip_{k-1}\circ \m_{k-2} + (1-\beta_{1,k}) (\sVtorip_k - \sVtorip_{k-1}) \circ \seg_{k}\\
    \nonumber
    &+ (1-\beta_{1,k})\sVtorip_{k-1}\circ (\srg_{k} +\beps_k ) -(1-\beta_{1,k-1})\sVtorip_{k-1} \circ ( \srg_{k-1}+\beps_{k-1} ) 
    \\ \nonumber 
     & = \beta_{1,k} \sVtorip_k\circ \m_{k-1} - \beta_{1,k-1} \sVtorip_{k-1}\circ \m_{k-2} + (1-\beta_{1,k}) (\sVtorip_k - \sVtorip_{k-1}) \circ \seg_{k}
     \\ \nonumber
    &+ (1-\beta_{1,k})\sVtorip_{k-1}\circ \beps_k - (1-\beta_{1,k-1})\sVtorip_{k-1} \circ \beps_{k-1} \\
    &+ (1-\beta_{1,k-1}) \sVtorip_{k-1} \circ  (\srg_{k} - \srg_{k-1})
     + (\beta_{1,k-1} - \beta_{1,k}) \sVtorip_{k-1} \circ \srg_{k}.
\end{align}
Next, we focus on providing upper bounds for $$R_{2,2,k}=\eta \|\sVtori_{k-1}^{\frac{1}{4}} \circ (\sdi_{k}-\sdi_{k-1})\|^2 \leq \eta G_{\infty} \|\sdi_{k}-\sdi_{k-1}\|^2 .$$
Observe that 
\begin{subequations}
\begin{align}\label{eq:rr2200}
\nonumber 
&\qquad  \left\|\beta_{1,k} \sVtorip_k\circ \m_{k-1}\right\|^2 + \left\|- \beta_{1,k-1} \sVtorip_{k-1}\circ \m_{k-2}\right\|^2 \\
& \leq 2 \max\left(\left\|\beta_{1,k} \sVtorip_k\circ \m_{k-1}\right\|^2 , \left\|- \beta_{1,k-1} \sVtorip_{k-1}\circ \m_{k-2}\right\|^2\right) \leq  \frac{2d\beta^2_{1,k-1}}{u_c},
\end{align} 
where the inequality follows from Lemma~\ref{VM_bound}. Using Remark~\ref{lem:hadamard}, we get
\begin{align}\label{eq:rr2220} 
\left\| (1-\beta_{1,k}) (\sVtorip_k - \sVtorip_{k-1}) \circ \seg_{k}\right \|^2 & \leq \|\seg_k\|_{\infty}^2  \|\sVtorip_k-\sVtorip_{k-1} \|_{1}^2 \leq G_{\infty}^2  \|\sVtorip_k-\sVtorip_{k-1} \|_{1}^2,
\end{align} 
where the last inequality uses Assumption~\ref{assumption:function}~(\ref{assumption:g_bounded}). Similarly,  
\begin{align}\label{eq:rr2240}
\nonumber 
\left\|(1-\beta_{1,k-1})\sVtorip_{k-1}\circ (\srg_{k-1}-\srg_{k})\right\|^2  & \leq \|\sVtorip_{k-1}\|_{\infty}^2 \|\srg_{k-1}-\srg_{k} \|^2  
\\
&\leq \frac{L^2}{G_0^2}  \| \z_{k-1} - \z_{k} \|^2, \\
\nonumber 
\|(\beta_{1,k}-\beta_{1,k-1})\sVtorip_{k-1}\circ \srg_{k}\|^2  & \leq  (\beta_{1,k}-\beta_{1,k-1})^2 \|\sVtorip_{k-1}\|_{\infty}^2 \left\|\srg_{k}\right\|^2 \\
&\leq \frac{(\beta_{1,k}-\beta_{1,k-1})^2 }{G_0^2}\|\srg_{k}\|^2 \label{eq:rr2260}.
\end{align} 
\end{subequations}
By taking the norm of \eqref{eq:ddif0}, using Remark~\ref{lem:hadamard} and~\eqref{eq:rr2200}--\eqref{eq:rr2260}, we get 
\begin{align}\label{eq:R2khalh0}
\nonumber \frac{R_{2,2,k}}{G_{\infty}}\leq \eta  \left\|\sdi_{k}-\sdi_{k-1} \right\|^2 &\leq 14 \eta d\beta^2_{1,k-1} u_c^{-1}+ 7 \eta G_{\infty}^2  \|\sVtorip_k-\sVtorip_{k-1} \|_{1}^2  \\
\nonumber
 & +  \frac{7 \eta L^2}{G_0^2}  \left(  \| \z_{k-1} - \z_{k} \|^2 \right) \\
 \nonumber
 & +  \frac{7 \eta}{G_0^2}  (\beta_{1,k}-\beta_{1,k-1})^2 \|\srg_{k}\|^2\\
  & +  \frac{7 \eta}{G_0^2} \left(\|\beps_{k}\|^2+ \|\beps_{k-1}\|^2\right).
\end{align}
By substituting~\eqref{eq:N_update0} and \eqref{eq:R2khalh0} into~\eqref{eq:R2k_0}, we obtain 
\begin{align*}
\nonumber
& \quad \left\|\z_k-\x_k\right\|^2 + \left\|\x_{k-1}-\z_k\right\|^2 \leq  \frac{R_{2,0,k}}{G_{0}} 
\\ \nonumber
&+ \frac{2 \eta }{G_{0}} \left(\beta_{1,k}DG_{\infty}\sqrt{du_c^{-1}} + DG_{\infty}^2 \|\sVtorip_k-\sVtorip_{k-1} \|_{1} + \Theta_{k}\right) \\
\nonumber 
& + \frac{14 \eta^2 G_{\infty}}{G_{0}}  \left(2d\beta^2_{1,k-1}u_c^{-1} +  G_{\infty}^2  \|\sVtorip_k-\sVtorip_{k-1} \|_{1}^2\right)  \\
\nonumber
 & +  \frac{14 \eta^2 G_{\infty}}{G_0^3} \left( L^2  \|\z_k - \z_{k-1}\|^2+ (\beta_{1,k}-\beta_{1,k-1})^2 \|\srg_{k}\|^2\right)\\
  & +  \frac{14 \eta^2 G_{\infty}}{G_0^3} \left(\|\beps_{k}\|^2+ \|\beps_{k-1}\|^2\right). 
\end{align*}
Now, summing the above inequality over $k$,  we obtain
\begin{align}\label{eq:rd22}
\nonumber 
& \quad  \left(1- \frac{28\eta^2 L^2 G_{\infty}}{G_0^3}\right)  \sum_{k=1}^{N} \left\|\x_{k-1}-\z_k\right\|^2 +  \left(1- \frac{28\eta^2 L^2 G_{\infty} }{G_0^3}\right) \sum_{k=1}^{N}\left\|\z_k-\x_k\right\|^2  
\\
\nonumber
& \leq  \frac{3 D^2 dG_{\infty}}{G_0}+ \frac{2 \eta }{G_0} \bigg(\frac{D\beta_{1,1} G_{\infty}}{1-\kappa} \sqrt{\frac{d}{u_c}}+ \frac{DG^2_{\infty}d}{G_0} + \sum_{k=1}^N \Theta_k\bigg) 
\\ 
\nonumber
& + \frac{28 \eta^2 d \beta_{1,1}^2 G_{\infty}}{u_c(1-\kappa^2)G_{0}} +  \frac{14 \eta^2d G_{\infty}^3}{G_0^3} \\
& + \frac{14\eta^2G_{\infty}}{G_0^3}  \sum_{k=1}^{N} (\beta_{1,k} - \beta_{1,k-1})^2 \|\srg_k\|^2 + \frac{28 \eta^2 G_{\infty}}{G_{0}^3} \sum_{k = 1}^N \|\beps_{k}\|^2 = :\text{R.H.S.}.
\end{align}
Here, we used Lemma~\ref{lem:v_diff} and  the fact that 
\begin{align}\label{sumd10} 
\nonumber 
 \sum_{k=1}^{N} \|\z_k - \z_{k-1}\|^2
& \leq  2 \sum_{k=1}^{N}\|\z_k - \x_{k-1}\|^2+  2 \sum_{k=1}^{N} \|\x_{k-1}- \z_{k-1}\|^2\\
& = 2 \sum_{k=1}^{N}\|\z_k - \x_{k-1}\|^2+  2 \sum_{k=1}^{N} \|\x_{k}- \z_{k}\|^2, 
\end{align}
where the inequality follows from Remark~\ref{lem:hadamard} and the equality uses our assumption $\x_0= \z_0=0$.  

Now, by our choice of step size $\eta$ in the beginning of Step~3, we have $ 1-(28\eta^2 L^2 G_{\infty})/G_0^3 \geq 1/2.$ Thus, \eqref{eq:r120} together with \eqref{eq:rd22} implies that 
\begin{align}
\nonumber 
 \frac{1}{N}\sum_{k=1}^{N} R_{2,k} &= \sum_{k=1}^{N} \left(\left\|\x_{k-1}-\z_k\right\|^2 + \left\|\z_k-\x_k\right\|^2 \right)  \leq \frac{2}{N} \text{R.H.S.},  
\end{align}
which gives \eqref{eq:stp2:result}.

\item[\textbf{Step 4}] (Convergence Analysis) In this step, we combine the results from the previous steps to establish an error bound for $N^{-1}\sum_{k = 1}^N \E\|\srg_k\|^2$. To do so, by substituting  \eqref{eq:stp2:result} and \eqref{R1k0} into \eqref{eq:thm_main0} and simplifying the terms, we obtain    
\begin{align}\label{eq:tabb}
\nonumber
\eta^2 (1-\beta_{1,1})^2 G_{\infty}^{-2}\frac{1}{N}\sum_{k = 1}^N \E\|\srg_k\|^2 & \leq  \frac{ 6 \eta^2 d \beta_{1,1}^2}{Nu_c (1-\kappa^2)} + \frac{6\eta^2 \sigma^2}{m G_0^2}  + \frac{18 D^2 dG_{\infty}}{NG_0}\\
\nonumber
&+ \frac{12 \eta D}{NG_0} \bigg(\frac{\beta_{1,1} G_{\infty}}{1-\kappa} \sqrt{\frac{d}{u_c}}+ \frac{G^2_{\infty}d}{G_0}\bigg) 
\\ 
\nonumber
& + \frac{168 \eta^2 d \beta_{1,1}^2 G_{\infty}}{Nu_c(1-\kappa^2)G_{0}} +  \frac{84 \eta^2d G_{\infty}^3}{NG_0^3} \\
& + \frac{1}{N}\sum_{i = 1}^N\frac{84 (1-\kappa) \beta_{1,1}^2 \eta^2G_{\infty}  }{G_0^3 \kappa^2(1+\kappa) }\E\|\srg_k\|^2 + \frac{168 \eta^2 G_{\infty}\sigma^2}{G_{0}^3m}.
\end{align}
Here, we used the fact that 
$$ \E \left[\sum_{k=1}^N \Theta_k\right] =0, \quad \text{and}  \quad \E \left[\sum_{k=1}^N \|\epsilon_k\|^2\right] =\frac{\sigma^2}{m},$$ by Assumptions~\ref{assumption:g_unbi} and~\ref{assu:bounded_var}, respectively.  

Now, it follows from~\eqref{eq:tabb}  that  
\begin{align*}
&\frac{B_0}{N} \sum_{k = 1}^N \E\|\srg_k\|^2  \leq \frac{B_1}{N} + \frac{B_2 \sigma^2 }{m},
\end{align*}
where 
\begin{align}\label{eq:const0}
\nonumber
& B_0 :=  \eta^2 (1-\beta_{1,1})^2 G_{\infty}^{-2}- \frac{84 (1-\kappa) \beta_{1,1}^2 \eta^2G_{\infty}}{G_0^3 \kappa^2(1+\kappa) },
\\
\nonumber 
& B_1:=    \frac{ 6 \eta^2 d \beta_{1,1}^2}{u_c (1-\kappa^2)} +\frac{18 D^2 dG_{\infty}}{G_0} + \frac{168 \eta^2 d \beta_{1,1}^2 G_{\infty}}{u_c(1-\kappa^2)G_{0}}
\\ 
\nonumber
& \qquad  + \frac{84 \eta^2d G_{\infty}^3}{G_0^3} + \frac{12 \eta D}{G_0} \bigg(\frac{\beta_{1,1} G_{\infty}}{1-\kappa} \sqrt{\frac{d}{u_c}}+ \frac{G^2_{\infty}d}{G_0}\bigg), 
\\ 
&  B_2:= \frac{6 \eta^2}{{G_0^2}} + \frac{168 \eta^2 G_{\infty}}{G_{0}^3}.
\end{align}

Next, define $C = \frac{(1+\kappa)\kappa^2 G_0^3}{168(1-\kappa) G^3_{\infty}}$ and pick $\beta_{1,1}\leq \frac{\sqrt{C}}{\sqrt{C} + 1}$. Then, dividing both sides by $B_0$ gives us the desired result. 
\end{enumerate}
\end{proof}

\begin{figure}%
    \centering
    \subfloat[\centering \ADAM]{{\includegraphics[width=5cm]{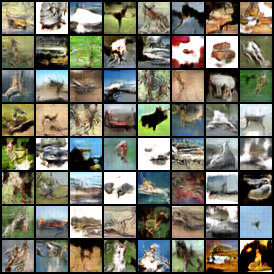} }}%
    \qquad
    \subfloat[\centering OAdagrad]{{\includegraphics[width=5cm]{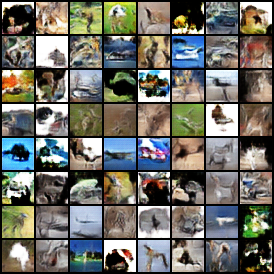} }}%
    \caption{Generated CIFAR-10 Samples}%
    \label{fig:sample}%
\end{figure}




\end{document}